\documentclass{amsart}
\usepackage{amsmath,amssymb}
\usepackage[utf8]{inputenc}
\usepackage{esvect}
\usepackage{amssymb, mathrsfs}
\usepackage{enumerate}
\usepackage[colorinlistoftodos]{todonotes}
\usepackage{verbatim}
\usepackage{mathrsfs}
\usepackage{hyperref}
\usepackage{amsthm, url}
\usepackage{xcolor}
\usepackage{tikz,pgf}
\usepackage{tikz-cd}
\usetikzlibrary{decorations.text}

\newtheorem{theorem}{Theorem}[section]
\newtheorem*{fibering lemma}{Fibering Lemma}
\newtheorem*{decomposition lemma}{Decomposition Lemma}
\newtheorem*{hurewicz theorem}{Hurewicz Theorem}
\newtheorem*{theorem A}{Theorem A}
\newtheorem{lemma}[theorem]{Lemma}
\newtheorem{proposition}[theorem]{Proposition}
\newtheorem{corollary}[theorem]{Corollary}

\theoremstyle{definition}
\newtheorem{definition}[theorem]{Definition}

\newtheorem{example}[theorem]{Example}

\newtheorem{question}[theorem]{Question}

\theoremstyle{remark}
\newtheorem*{proofA}{Proof of Theorem A}

\renewcommand{\vec}{\mathbf}
\newcommand{\U}[1]{{\mathcal{U}_{#1}}}
\newcommand{\V}[1]{{\mathcal{V}_{#1}}}
\newcommand{\W}[1]{{\mathcal{W}_{#1}}}
\newcommand{\sE}{\mathscr{E}}
\newcommand{\sF}{\mathscr{F}}

\newcommand{\df}[1]{{{\em #1}}}

\newcommand{\NN}{\mathbb{N}}
\newcommand{\ZZ}{\mathbb{Z}}

\renewcommand{\epsilon}{\varepsilon}

\DeclareMathOperator{\as}{asdim}
\DeclareMathOperator{\con}{con}

\DeclareMathOperator{\ord}{ord}

\makeatother

\begin{document}

\title{Coarse free products}
\author{G.~Bell}
\address{Department of Mathematics and Statistics, The University of North Carolina at Greensboro, Greensboro, NC 27412, USA}
\email{gcbell@uncg.edu}

\author{A.~Lawson}
\address{Department of Mathematics and Statistics, The University of North Carolina at Greensboro, Greensboro, NC 27412, USA}
\email{azlawson@uncg.edu}
\date{\today}

\begin{abstract}
We define a notion of free product for coarse spaces that generalizes the corresponding notion of a free product for groups. We show that free products preserve coarse properties such as coarse property C, finite coarse decomposition complexity, and coarse property A. We also give an upper bound estimate on the dimension of a coarse free product in terms of the dimension of its factors.
\end{abstract}

\maketitle

\section{Introduction}
The free product of groups arises in the computation of fundamental groups of path-connected spaces expressed as a union of path-connected subspaces with simply connected intersection. Given groups $A=\langle S_A\mid R_A\rangle$ and $B=\langle S_B\mid R_B\rangle$, the free product is the group $A\ast B$ with presentation $\langle S_A\sqcup S_B \mid R_A\sqcup R_B\rangle$. Given a class of groups it is natural to ask whether this class is closed under the operation of free product. This motivated the study of asymptotic dimension ($\as$) of free products~\cite{bell-dranishnikov2001}, where the Bass-Serre theory was applied to prove that if $A$ and $B$ are two groups with $\as A\le n$ and $\as B\le n$, then $\as A\ast B\le 2n+1$. This upper bound was also shown to hold for a metric notion of free product of metric spaces. Later, the upper bound estimate for the asymptotic dimension of a free product of groups was sharpened~\cite{BDK} and a formula was established for the asymptotic dimension of (the more general) free product with amalgamation~\cite{Dranish-amalgams}. 

Other coarse properties of groups, including finite decomposition complexity~\cite{gty2013}, property A~\cite{bell2003}, and asymptotic property C~\cite{bell-nagorko2016}, are preserved by the operation of free product. In the cases of finite decomposition complexity and property A, the method of proof has become standard: one applies the Bass-Serre theory, union permanence, and fibering permanence, (see~\cite{guentner2014}). Other techniques were required in the case of asymptotic property C~\cite{bell-nagorko2016}.

Motivated by our work with coarse direct products~\cite{BL} and the metric notions of free product described by Bell and Dranishnikov~\cite{bell-dranishnikov2001} and Bell and Nagórko~\cite{bell-nagorko2016}, we define a coarse free product of coarse spaces. We show that our coarse free product preserves finite dimension, coarse property A, coarse notions of decomposition complexity, and coarse property C. We also provide an upper bound for the dimension of a coarse free product in terms of the dimension of its factors. Our coarse free product is analogous to the free product of discrete metric spaces, but is distinct from the metric free product when applied to general metric spaces (see Example~\ref{ex:strictinequality}).

The paper is organized as follows. In the next section we give basic definitions related to the coarse category. We then define the coarse free product in Section~\ref{sec:coarsefreeproduct}. In Section~\ref{sec:permanenceviaunions}, we use techniques akin to Guentner's~\cite{guentner2014} to exhibit permanence of many coarse properties with respect to the coarse free product. In Section~\ref{sec:asdim} we provide an upper bound for the dimension of a coarse free product. We conclude by showing how the arguments of Bell and Nagórko~\cite{bell-nagorko2016} can be adapted to show that coarse property C is preserved by coarse free products in Section~\ref{sec:Property C}.

The authors are grateful to Elisa Hartmann for pointing out two errors in a previous version of this paper.  The definition of coarse fiber and the proofs of Lemma~\ref{lem:Guentner} and \hyperref[ThmA]{Theorem~A} have been changed to address these errors.

\section{Preliminaries}\label{sec:preliminaries}

To describe Roe's coarse category, we first recall the operations of composition and inverse on the pair groupoid. Let $X$ be any set. 
The \df{composition} of two subsets $E\subset X\times X$ and $F\subset X\times X$ is the (possibly empty) set $E\circ F := \{(x,z)\mid \exists y\in X, (x,y)\in E, (y,z)\in F\}$. The \df{inverse} of $E\subset X\times X$ is the set $E^{-1} := \{(y,x)\mid (x,y)\in E\}$. The \df{diagonal} is the set $\Delta\subset X\times X$ defined by $\Delta:= \{(x,x)\in X\times X\}$. It will be convenient to allow points in our spaces to be infinitely far apart; we use the term $\infty$-\df{metric} to mean a ``metric'' that can take the value $+\infty$.

\begin{definition}\cite[Definition 2.3]{roe2003}
Let $X$ be a set and let $\sE$ be a collection of subsets of $X\times X$. We say that $\sE$ is a \df{coarse structure} on $X$ (and call the pair $(X,\sE)$ a \df{coarse space}) if $\sE$ contains the diagonal, and is closed under inverses, finite unions, subsets, and compositions.
The elements of $\sE$ are called \df{entourages}.\end{definition}

\begin{definition} Let $(X,\sE)$ be a coarse space. Let $\U{}$ be a collection of subsets of $X$. Let $E\in \sE$. We say that the family $\U{}$ is $E$-\df{disjoint} if $(U\times U')\cap E=\emptyset$ whenever $U\neq U'$ are elements of $\U{}$. Let $K\in\sE$. We say that $\U{}$ is $K$-\df{bounded} if $\cup_{U\in \U{}}U\times U\subset K$. We call $\U{}$ bounded if it is $K$-bounded for some entourage $K$. A subset $B$ of the coarse space $(X,\sE)$ will be called $K$-\df{bounded} if $\{B\}$ is $K$-bounded for some entourage $K$.
\end{definition}

Given an entourage $E$, we write $E^k$ to mean the $k$-fold composition $E\circ E\circ\cdots\circ E$. By $E^0$, we mean $E\cap\Delta$.  We define $D_E: X\times X \to \mathbb{Z}_{\ge 0} \cup \{+\infty\}$ by $D_E(x,y) = \min\{k \colon (x,y)\in  E^k\}$.

\begin{proposition}~\cite[Proposition 3.6]{bell-moran-nagorko2016}\label{Distance} 
Let $(X,\sE)$ be a coarse space and fix an entourage $E$ that is \df{symmetric} in the sense that $E=E^{-1}$. Then we have that 
\begin{enumerate}
	\item $D_E$ is symmetric;
	\item $D_E(x,y) = 0$ iff $x = y$;
	\item $D_E(x,y) \leq D_E(x,z) + D_E(z,y)$; and
    \item for every $w, z \in X$ and each $A \subset X$ we have      \[ |D_E(w,A) - D_E(z, A)| \leq D_E(w,z),\]
    where $D_E(x, A) = \inf \{ D_E(x,a) \colon a \in A \}$.
\end{enumerate}
\end{proposition}

Let $(X,\sE)$ and $(Y,\sF)$ be coarse spaces and let $f$ be a map $f:X\to Y$. The map $f$ is said to be \df{uniformly expansive} if for each $E\in\sE$ we have $(f\times f)(E) \in \sF$. Two maps $f,f':X\to Y$ are said to be \df{close} if $\{(f(x),f'(x))\mid x\in X\}\in\sF$. Finally, the coarse spaces $X$ and $Y$ are said to be coarsely equivalent if there are uniformly expansive\footnote{In the present setting there is no need to require that such maps are also \df{proper}, i.e. that the inverse of each bounded set it bounded, cf.~\cite{roe2003}.} maps $f:X\to Y$ and $g:Y\to X$ such that the compositions in both directions are close to the identity maps. A \df{coarse fiber} of $f$ at scale $F\in\sF$ is any set $A\subset X$ such that there exists a $y\in Y$ such that $f(A)$ is contained in the \df{$F$-ball} $F^{-1}[\{y\}]$; i.e., there is a $y\in Y$ such that for every $a\in A$, $(y,f(a))\in F$.

It will be convenient to describe properties of families of coarse spaces that are satisfied uniformly. To achieve this, we use the notion of total space, following Guentner~\cite{guentner2014}.

\begin{definition}
Let $(X,\sE)$ be a coarse space. Let $\U{} = \{U_{i}\}_{i\in J}$ be a collection of subsets of $X$. We define a coarse space called the \df{total space} $T(\U{};J)$ of the $\{U_i\}_{i\in J}$ as follows. The underlying set is the disjoint union, $\bigsqcup_{i\in J} U_i$. The entourages are the disjoint unions \[\bigsqcup_{i\in J} \big(E\cap (U_i\times U_i)\big)\subset \bigsqcup_{i\in J}\left(U_i\times U_i \right),\] where $E$ ranges over $\sE$.

Let $\mathscr{P}$ be a \df{coarse property}; i.e., a property of a coarse space that is invariant of coarse equivalence. We say that the family $\{U_i\}_{i\in J}$ has property $\mathscr{P}$ \df{uniformly} if the total space $T(\U{};J)$ has $\mathscr{P}$.
\end{definition}

\section{The Coarse Free Product} \label{sec:coarsefreeproduct}
Let $(X,\sE)$ be a coarse space. Fix a basepoint $x_0\in X$. Consider the collection $\ast X$ of words in the alphabet $X\setminus\{x_0\}$ along with the empty word $\boldsymbol{\epsilon}$, which we identify with $x_0$. By way of notation, letters set in bold type will be elements of $\ast X$. We define the concatenation $\mathbf{x}\cdot \mathbf{x'}$ of two words in the usual way, so that $\mathbf{x}\cdot\boldsymbol{\epsilon}=\boldsymbol{\epsilon}\cdot\mathbf{x}=\mathbf{x}$. We will often write concatenation as juxtaposition unless extra emphasis is needed. We will also allow words in $\ast X$ and elements $x\in X$ to be concatenated to form new words; i.e., we do not distinguish between words consisting of a single letter and the letters themselves. For each distinct pair of elements $\mathbf{x}$ and $\mathbf{x'}\in \ast X$ there is a unique $\mathbf{a}\in\ast X$ with the properties that $\mathbf{x}=\mathbf{a}b\mathbf{c}$ and $\mathbf{x}'=\mathbf{a}b'\mathbf{c'}$ with $b\neq b'$, $b, b'$ both in $X$, and $\mathbf{c}, \mathbf{c'}$ both in $\ast X$. (Note that we allow the words $\mathbf{a},\vec{c}$ or $\vec{c'}$ to be $\boldsymbol{\epsilon}$ and at most one of $b$ or $b'$ may be $x_0$ when the corresponding $\vec c$ or $\vec c'$ following it is $\boldsymbol{\epsilon}$.)  When necessary, we use the notation $\mathbf{a}=\mathbf{x}\wedge \mathbf{x'}$ for this element. 

Let $E\in\sE$ be given. Define $\|\boldsymbol{\epsilon}\|_E=\|\boldsymbol{\epsilon}\|^E=0$; for nonempty $\vec{x} = x_1x_2\cdots x_k\in \ast X$ define $\|\vec{x}\|_E = \sum_{i=1}^k D_E(x_0, x_i) $ and $\|\vec{x}\|^E = \sum_{i=1}^k D_E(x_i, x_0) $. In the case that $E$ is symmetric, $\|\vec{x}\|_E = \|\vec{x}\|^E$. Moreover define $D^*_E : \ast X\times \ast X\to \ZZ_{\ge 0}\cup\{\infty\}$ by \[D^*_E(\vec{x},\vec{x}') =\begin{cases} 0 & \vec{x}=\vec{x}'\\ D_E(b,b') + \|\vec{c}\|_E + \|\vec{c}'\|^E& \vec{x}\neq\vec{x}'.\end{cases}\] 

\begin{definition}
Let $(X,\sE)$ be a coarse space, let $E\in\sE$, and let $x\in X$. We define the \df{symmetric ball of size $E$ about $x$} to be $B(x,E) = E_x\cup E^x$ where $E_x = \{y\in X\mid (x,y)\in E\}$ and $E^x = \{y\in X\colon (y,x)\in E\}$
\end{definition}

\begin{proposition}
If $E\in \sE$ is symmetric, then $D^*_E$ is an $\infty$-metric on $\ast X$. \qed
\end{proposition}
\begin{proposition}
Let $n\in \mathbb{Z}_{\ge 0}$ and put $\langle E,n\rangle=\{(\vec{x},\vec{x}')\in\ast X\times \ast X\colon D_E^*(\vec{x},\vec{x}')\le n\}$. Define the collection $\ast\sE$ to be the subset closure of $\{\langle E,n\rangle\colon E\in\sE, n\in\mathbb{N}\}$. Then, $\ast\sE$ is a coarse structure on $\ast X$.
\end{proposition} 

\begin{proof} 
We must show that $\ast\sE$ (a) contains the diagonal, (b) is closed under inverses, (c) is closed finite under unions, (d) is closed under subsets, and (e) is closed under compositions.

(a) It is clear that $\langle\Delta,0\rangle$ contains the diagonal in $\ast X\times \ast X$.

(b) Given $L\in \ast \sE$, we take a symmetric $E$ such that $\langle E,n\rangle$ contains $L$. It is easy to see that $\langle E^{-1},n\rangle=\langle E,n\rangle^{-1}$.

(c) Let $L$ and $L'$ be given elements of $\ast \sE$. Find $E, E'\in\sE$ and $n,n'\in\mathbb{N}$ such that $L\subset \langle E,n\rangle$ and $L'\subset \langle E',n'\rangle$. Then $L\cup L'\subset \langle E,n\rangle\cup\langle E', n'\rangle\subset \langle E\cup E',n+n'\rangle$.

(d) This holds by definition.

(e) Let $L$ and $L'$ be given elements of $\ast \sE$. Find $E, E'\in\sE$ and $n,n'\in\mathbb{N}$ such that $L\subset \langle E,n\rangle$ and $L'\subset \langle E',n'\rangle$. Then, $\langle E,n\rangle\circ \langle E',n'\rangle\subset \langle E\cup E',n+n'\rangle$.

\end{proof}

\begin{definition}
Let $(X,\sE)$ be a coarse space with basepoint $x_0\in X$. The coarse space $(\ast X, \ast\sE)$ constructed above is called the \df{coarse free product of $(X,\sE)$}.
\end{definition}

For ease of notation, we defined the unary free product, following Bell and Nagórko~\cite{bell-nagorko2016}. One can obtain the free product $X\ast Y$ of pointed spaces $(X,x_0)$ and $(Y, y_0)$ as the subset of words whose letters alternate between $X$ and $Y$ inside of $\ast (X\sqcup Y/x_0\sim y_0)$.

\begin{definition} \label{def:order}
Let $(X,\sE)$ be a coarse space with basepoint $x_0\in X$. Let $\vec x\neq \boldsymbol{\epsilon}$. We say the \df{order (or length)} of $\vec{x}$ is $n$ (and write $\ord(\vec x)=n$), if $\vec x$ can be expressed as $\vec{x} = x_1x_2\cdots x_n\in\ast X$ with $x_i\neq x_0$, for all $i$. We define $\ord(\boldsymbol{\epsilon})=0$.
\end{definition}

Next, we wish to compare the coarse free product with the metric free product. A metric space $(X,d)$ carries a natural coarse structure called the \df{bounded coarse structure}, in which entourages are subsets of the form $\{(x,y)\in X\times X\colon \sup d(x,y)<\infty\}$~\cite{roe2003}.

Suppose $(X,d)$ is a metric space and $x_0\in X$. The (metric) free product $\ast X$ with respect to $x_0$ was defined by Bell and Nagórko~\cite{DecompThm}, see also~\cite{bell-dranishnikov2001}. They showed that the function $d^*:\ast X\times \ast X\to[0,\infty)$ is a metric, where $d^\ast$ is defined by $d^\ast(\vec{x},\vec{x})=0$ and if $\vec x\neq \vec x'$ are expressed as $\vec x=\vec abx_1\cdots x_n$, $\vec x'=\vec ab'x_1'\cdots x_m'$, then \[d^*(\vec{a}bx_1\ldots x_n, \vec{a}b'x_1'\ldots x_m') = d(b,b') + \sum_{i=1}^n d(x_i, x_0) + \sum_{j=1}^m d(x_j,x_0).\] 

\begin{proposition}\label{prop:metricvcoarse}
Suppose $(X, d)$ is a metric space. Let $\mathscr{B}_d$ represent the bounded coarse structure on $X$ associated to $d$, let $d^*$ represent the metric on the free product $\ast X$, and let $\mathscr{B}_{d^*}$ represent the bounded coarse structure associated to $d^*$. The coarse free product structure $\ast\mathscr{B}_d$ is finer than the bounded coarse space $\mathscr{B}_{d^*}$. That is $\ast\mathscr{B}_d\subset \mathscr{B}_{d^*}$. In the case that $(X,d)$ is a uniformly discrete (see below) metric space, the two structures are equal.
\end{proposition}
\begin{proof}

Let $L\in\ast\mathscr{B}_d$. We may assume $L$ is a subset of an entourage of the form $\langle E,n\rangle$ where $E \subseteq \{(y,y')\in X\times X\mid d(y,y') < R\}$ for some $n\in\mathbb{N}$ and $R\in\mathbb{R}$. Suppose $(\vec{x},\vec{x}')\in L$. Then, with $\vec a=\vec{x}\wedge \vec{x}'$,  \[n\ge D_E^*(\vec{x},\vec{x}') = D_E^*(\vec{a}b\vec{c}, \vec{a}b'\vec{c}')=D_E(b,b')+\|\vec{c}\|_E+\|\vec{c}'\|^E,\] where $\ord(\vec{c})< n$ and $\ord(\vec{c}')< n$. But this means that $d^*(\vec{x},\vec{x}')< 3Rn$ thus $L\in\mathscr{B}_{d^*}$. 

Suppose now that $(X,d)$ is a uniformly discrete metric space; i.e., suppose that $r = \inf\{d(y,y')\mid y\neq y'\in X\}>0$. Let $F\in \mathscr{B}_{d^*}$; then there is some $R\in\mathbb{R}$ such that $F \subseteq \{(\vec{x},\vec{x}')\in\ast X\times\ast X\mid d^*(\vec{x},\vec{x}')< R\}$. Take $k\in\mathbb{N}$ such that $(k-1)r\le R<kr$. With this $k$, we see that if $(\vec{x},\vec{x}') = (\vec{a}b\vec{c}, \vec{a}b'\vec{c}')\in F$ then $\ord(b\vec{c}) \le k$ and $\ord(b'\vec{c}')\le k$. Thus $F\subseteq \langle E,k\rangle\in\ast\mathscr{B}_d$, where $E=\{(x,x')\in X\times X\mid d(x,x')<R\}$. 
\end{proof}

\begin{example}\label{ex:strictinequality}
Suppose that $(X,d)$ is a metric space with basepoint $x_0$ in which there exists a sequence of distinct points $\{x_i\}_i$ converging to $x_0\in X$ with the property that $\sum_id(x_0,x_i)<1$. Then, all elements of the sequence of pairs $\{(x_0,x_0\cdot x_1\cdots x_i)\}_i$ belong to $\{(\vec{x},\vec{x}')\mid d^\ast(\vec{x},\vec{x}')\le 1\}$, but there is no $E\in \mathscr{B}_d$ such that $\{(x_0,x_0\cdot x_1\cdots x_i)\}_i$ is in $\langle E,k\rangle$ for any fixed $k$. Thus, the inclusion in Proposition~\ref{prop:metricvcoarse} is strict.
\end{example}

The sequence in Example~\ref{ex:strictinequality} shows that the identity map is not a coarse equivalence between the bounded coarse structure on the metric free product and the coarse free product taken with respect to the bounded coarse structure on $X$. We conjecture that in cases such as Example~\ref{ex:strictinequality} no coarse equivalence exists.

\begin{question}
Let $(X,d)$ be a metric space. Suppose that the spaces $(\ast X,\ast\mathscr{B}_d)$ and $(X,\mathscr{B}_{d^*})$ coarsely equivalent. Does it follow that $X$ is uniformly discrete? 
\end{question}

\section{Free-product permanence via unions and fibering} \label{sec:permanenceviaunions}

In this section, we show that several coarse properties are preserved by coarse free products. Our approach is similar to the one given by Guentner~\cite{guentner2014} for metric spaces. Indeed, the only challenge is to replace the real parameters describing the metric disjointness and the metric uniform diameter bound with suitably chosen entourages and generators for these entourages. 

We consider a property $\mathscr{P}$ that is a coarse invariant, e.g. finite asymptotic dimension, or finite coarse decomposition complexity. We prove that such a property $\mathscr{P}$ is preserved by the coarse free product construction whenever trees have $\mathscr{P}$ and it satisfies so-called excisive union permanence and fibering permanence, described below.

\begin{definition} Let $\mathscr{P}$ be a coarse property. Suppose that $(X,\sE)$ is a coarse space. We say that $\mathscr{P}$ satisfies \df{excisive union permanence} if, whenever $X$ is expressed as a union $X=\bigcup_{\alpha\in J} X_\alpha$ such that
\begin{enumerate}
\item $\{X_\alpha\}_{\alpha\in J}$ has uniform $\mathscr{P}$ and 
\item for every $E\in \sE$ there is a $Y_E\subset X$ with $\mathscr{P}$ such that $\{X_\alpha\setminus Y_E\}_{\alpha\in J}$ is $E$-disjoint,
\end{enumerate}
it follows that $X$ has $\mathscr{P}$.

\end{definition}
\begin{definition} Let $\mathscr{P}$ be a coarse property. Suppose that $(X,\sE)$ is a coarse space. We say that $\mathscr{P}$ satisfies \df{fibering permanence} if, whenever $f:X\to Y$ is a uniformly expansive map to a coarse space $(Y,\sF)$ such that $Y$ has $\mathscr{P}$ and for each $F\in\sF$ the coarse fibers of $f$ at scale $F$ have $\mathscr{P}$ uniformly, it follows that $X$ has $\mathscr{P}$.
\end{definition}

\begin{theorem A} \label{ThmA}
Let $(X, \sE)$ be a coarse space with basepoint $x_0$. Let $\mathscr{P}$ be a coarse property of coarse spaces that satisfies excisive union permanence and fibering permanence. Suppose trees have $\mathscr{P}$. Then, the coarse free product $(\ast X, \ast\sE)$ has property $\mathscr{P}$ whenever $(X,\sE)$ does. 
\end{theorem A}

\begin{lemma}\label{lem:Guentner}
Let $(X,\sE)$ be a coarse space with basepoint $x_0\in X$ and consider the free product $(\ast X, \ast\sE)$. Let $\mathscr{P}$ be a coarse property. Suppose that $A\subset \ast X$ has $\mathscr{P}$. Then, for any subset $Y\subset \ast X$, the family $\{\vec{y}\cdot A\}_{\vec{y}\in Y}$ has property $\mathscr{P}$ uniformly.
\end{lemma}

\begin{proof}
We form the total space $T(A;Y)$ of the family $\{A\}_{\vec{y}\in Y}$ (indexed by $\vec{y}\in Y$) and observe that this total space has property $\mathscr{P}$.  Note that we write $T(A,Y)$ to mean $T(\{A\},Y)$.

We define a map $f:T(\vec{y}\cdot A;Y)\to T(A;Y)$ by $(\vec{y}\cdot\vec{a},\vec{y})\mapsto (\vec{a},\vec{y})$. It remains to show that this is a coarse equivalence.  We prove one direction of the equivalence and note that this follows from the fact that for any $K\in\sE$ and any elements $\vec{y}, \vec{a}, \vec{a}'\in \ast X$ we have $D_K^\ast(\vec{a},\vec{a}')=D_K^\ast(\vec{y}\cdot\vec{a},\vec{y}\cdot\vec{a}').$ The other direction is analogous.

An entourage $F$ in the space $T(\vec{y}\cdot A; Y)$ has the form $\bigsqcup_{\vec{y}\in Y}\big(E\cap(\vec{y}\cdot A\times\vec{y}\cdot A)\big)$, where $E\in\ast\sE$. Suppose that $K\in \sE$ and $n\in\NN$ have the property that $E\subset\langle K,n\rangle$. We claim that $(f\times f)(F)\subset \bigsqcup_{\vec{y}\in Y}\langle K,n\rangle \cap (A\times A)$, which is an entourage in $T(A;Y)$. Indeed, if $(\vec x, \vec x')\in F$, then there is some $\vec y\in Y$ such that $\vec x=\vec y\cdot \vec a$ and $\vec x'=\vec y\cdot \vec a'$ for some $\vec a$ and $\vec a'$ in $A$. Since $E\subset \langle K, n\rangle$, we see that $n\ge D^\ast_K(\vec x,\vec x')=D^\ast_K(\vec a, \vec a')$, so that $(\vec a, \vec a')\in \langle K, n \rangle$. Thus, $(f\times f)(\vec x, \vec x')\in \bigsqcup_{\vec{y}\in Y}\langle K,n\rangle \cap (A\times A)$, as required.
\end{proof}

\begin{lemma}\label{lem:concat}
Let $\mathscr{P}$ be a coarse property of coarse spaces that satisfies excisive union permanence. Let $(X,\sE)$ be a coarse space with basepoint $x_0\in X$. Put $X^{(n)}=\{\vec x\in \ast X\colon \ord(\vec x)=n\}$ (see Definition~\ref{def:order}). If $X$ has property $\mathscr{P}$, then for each $n\in\NN$, $X^{(n)}$ has property $\mathscr{P}$. 
\end{lemma}

\begin{proof}
We apply induction. Put $X^*=X\setminus\{x_0\}$. Observe that $X^{(1)}= X^\ast$. Suppose the conclusion holds for $X^{(n-1)}$. Observe that $X^{(n)}=\bigcup_{\mathbf{x}\in X^{(n-1)}} \mathbf{x}\cdot X^\ast.$ 
By Lemma~\ref{lem:Guentner}, the collection $\{\mathbf{x}\cdot X^*\}_{\mathbf{x}\in X^{(n-1)}}$ has $\mathscr{P}$ uniformly.

Let $L\in\ast\sE$ be given. Then, we can find a symmetric entourage $E\in\sE$ and a natural number $m$ such that $L\subset \langle E,m\rangle$. 
Put $Y_{\langle E,m\rangle}=X^{(n-1)}\cdot (B(x_0, E^m)\cap X^\ast)$. It is straightforward to show that $Y_{\langle E,m\rangle}$ is coarsely equivalent to $X^{(n-1)}$. Thus we see that $Y_{\langle E,m\rangle}$ has property $\mathscr{P}$ by the inductive hypothesis. Finally, we show that the collection $\{\mathbf{x}\cdot X^\ast\setminus Y_{\langle E,m\rangle}\colon \vec x\in X^{(n-1)}\}$ is $\langle E,m\rangle$-disjoint. To this end if $\mathbf{w}\neq \mathbf{w}'$ are in this collection then we may write $\mathbf{w} = \vec a\cdot b\cdot \vec c\cdot x$ and $\mathbf{w}' = \vec a\cdot b'\cdot \vec c'\cdot x'$ where $\vec a,\vec c,\vec c'\in \ast X$, $b,b',x,x'\in X^\ast$, $b\neq b'$ and $x,x'\notin E^m$. Then \[D^*_{E}(\mathbf{w},\mathbf{w}') = D_E(b,b') +\|\vec c\|_E + \|\vec c'\|^E + D_E(x_0,x) + D_E(x',x_0)\ge m \]
by assumption since $x,x'\notin E^m$. Therefore the collection is $\langle E,m\rangle$-disjoint (hence $L$-disjoint).

We apply excisive union permanence to complete the proof. 
\end{proof}

\begin{proofA} Let $T$ be a graph whose vertex set is in one-to-one correspondence with the elements of $\ast X$. We denote by $t_\vec{x}$ the vertex of $T$ corresponding to the element $\vec{x}\in\ast X$. We connect two vertices $t_{\vec{x}_1}$ and $t_{\vec{x}_2}$ of $T$ by an edge if and only if there is an $x\in X$ ($x\neq x_0$) for which $\vec{x}_1x=\vec{x}_2$ or $\vec{x}_2x=\vec{x}_1$ (as elements of $\ast X$). It is clear that $T$ is a tree. Give $T$ the bounded coarse structure it inherits as a metric space with the edge-length metric.

Define $f:\ast X\to T$ by $f(\vec{x})=t_\vec{x}$. We claim that $f$ is uniformly expansive. To this end, let $L\in\ast \sE$ be given. Then, we can find an $E\in \sE$ and an $n\in\mathbb{N}$ such that $L\subset \langle E,n\rangle$. Suppose $\vec{x}\neq \vec{x}'$ and that $(\vec{x},\vec{x}')\in L$. Put $\vec{a}=\vec{x}\wedge \vec{x}'$, find $b\neq b'$ in $X$ and sequences $x_1,x_2,\ldots x_m$ and $x'_1,x'_2,\ldots, x'_{m'}$ of elements of $X^\ast$ such that $\vec{x}=\vec{a}bx_1\cdots x_m$ and $\vec{x}'=\vec{a}b'x'_1\cdots x'_{m'}$. 

Then, 
\[n\ge D^\ast_E(\vec{x},\vec{x}')=D_E(b,b')+\sum_{i=1}^m\|x_i\|_E+\sum_{i=1}^{m'}\|x_i'\|^E\ge 1+m+m'=d_T(t_\vec{x},t_{\vec{x}'})-1.\]
Thus, for all pairs $(\vec{x},\vec{x}')\in L$, we have $d_T(t_\vec{x},t_{\vec{x}'})\le n+1$. Therefore, the image $(f\times f)(L)$ is a uniformly bounded set, which means $f$ is uniformly expansive.  

Let $F$ be an entourage in the bounded coarse structure on $T$. Let $A$ be a coarse fiber of $f$ at scale $F$. Then there is a $\vec{y}\in\ast X$ such that for every $a\in A$, $(t_{\vec{y}},f(a))\in F$. Thus, there is an $R>0$ such that $d_T(t_{\vec{y}},f(a))\le R$ and so there is some $\vec y'\in \ast X$ such that $A\subset \vec{y}'\cdot X^{(\le 2R)}$, where $X^{(\le 2R)}$ denotes the collection of elements of $\ast X$ with order at most $2R$.

By Lemma~\ref{lem:Guentner} and Lemma~\ref{lem:concat}, coarse fibers of $f$ have $\mathscr{P}$ uniformly. Since $\mathscr{P}$ is assumed to satisfy fibering permanence, we are done. \qed 
\end{proofA}

We note that any tree (in the bounded coarse structure) has asymptotic dimension $1$ ~\cite{gromov93} and therefore has coarse property A, coarse property C, as well as finite weak coarse decomposition complexity, finite coarse decomposition complexity, and straight finite decomposition complexity~\cite{bell-moran-nagorko2016}.

Guentner shows that coarse property A satisfies fibering and excisive union permanence~\cite[Theorem 6.5, Theorem 6.3]{guentner2014}. Therefore, \hyperref[ThmA]{Theorem A} immediately implies:

\begin{corollary}
Let $(X,\sE)$ be a coarse space with coarse property A. Let $x_0$ be a basepoint. Then, the coarse free product $\ast X$ has coarse property A. 
\end{corollary}

Bell, Moran, and Nag\'orko show that $\mathscr{P}$ satisfies fibering permanence when $\mathscr{P}$ is finite weak coarse decomposition complexity, finite coarse decomposition complexity, or straight finite decomposition complexity~\cite[Theorem 4.14]{bell-moran-nagorko2016}. Moreover, each of these properties satisfy excisive union permanence~\cite[Theorem 4.18]{bell-moran-nagorko2016}.

\begin{theorem}
Let $\mathscr{P}$ be one of the coarse properties: finite weak coarse decomposition complexity, finite coarse decomposition complexity, or straight finite decomposition complexity. Then, $\mathscr{P}$ satisfies fibering permanence and excisive union permanence.
\end{theorem}

Combining this with \hyperref[ThmA]{Theorem A} immediately implies:

\begin{corollary}
Let $(X,\sE)$ be a coarse space with a property $\mathscr{P}$ among finite weak coarse decomposition complexity, finite coarse decomposition complexity, or straight finite decomposition complexity. Let $x_0$ be a basepoint. Then, the coarse free product $\ast X$ has $\mathscr{P}$. 
\end{corollary}

It is not known whether coarse property C satisfies fibering permanence, so we cannot use \hyperref[ThmA]{Theorem A} to show that coarse free products preserve coarse property C. We prove this using techniques similar to Bell and Nagórko~\cite{bell-nagorko2016} in Section~\ref{sec:Property C}.

\section{Asymptotic dimension of a free product}\label{sec:asdim}

By applying permanence results, we can show that finite asymptotic dimension is preserved by taking coarse free products as above. Instead, we apply the techniques of Theorem~A to find an upper bound for the asymptotic dimension of a coarse free product.

The asymptotic dimension of a metric space was defined by Gromov~\cite{gromov93}. Later, Roe~\cite{roe2003} and then Grave~\cite{grave2005} provided definitions of asymptotic dimension of coarse spaces as follows.

\begin{definition} 
Let $(X,\sE)$ be a coarse space.
We say that the \df{asymptotic dimension} of the coarse space $X$ does not exceed $n$ and write $\as X\le n$ if for every $E\in \sE$ there are families $\U{i}$ ($i=0,1,\ldots,n$) of subsets of $X$ such that 
\begin{enumerate}
	\item $\cup_{i=0}^n\U{i}$ covers $X$;
    \item each $\U{i}$ is $E$-disjoint; and 
    \item there is some $K\in\sE$ such that each $\U{i}$ is $K$-bounded.
\end{enumerate}
\end{definition}

The next definition can be deduced from previous ones, but we include it for clarity.

\begin{definition}
Let $n\in\mathbb{Z}_{\ge 0}$. We say that \df{coarse fibers of $f$ have asymptotic dimension of at most n uniformly} if for every $L\in \sE$ and $F\in\sF$ there is some $K\in\sE$ such that whenever $A$ is a coarse fiber at scale $F$, there exist families of subsets, $\U{0}, \U{1}, \ldots, \U{n}$, of $A$ that are $K$-bounded and $L$-disjoint, such that  $\U{0}\cup\U{1}\cup\ldots\cup\U{n}$ covers $A$.
\end{definition}

It is straightforward to show that the notion of uniform asymptotic dimension at most $k$ in the sense of Guentner agrees with the notion from Bell and Dranishnikov~\cite{bell-dranishnikov2001}, when both are suitably translated to the coarse category.

\begin{proposition}
Let $\U{}=\{U_i\}_{i\in J}$ be a collection of subsets of $(X,\sE)$. Then $T(\U{};J)$ has asymptotic dimension at most $k$ if and only if for every $E\in \sE$ there is a $K\in \sE$ and families $\V{0}^i,\V{1}^i,\ldots,\V{k}^i$ such that for each $i$, $\cup_j \V{j}^i$ covers $U_i$, each $\V{j}^i$ is $E$-disjoint, and each $\V{j}^i$ is $K$-bounded. \qed
\end{proposition}

We need the following union permanence result for asymptotic dimension.

\begin{theorem} \label{cFADUnion} \cite[Theorem 3.17]{bell-moran-nagorko2016}
Let $(X,\sE)$ be a coarse space. Suppose that $X=\bigcup_{\alpha}X_{\alpha}$, where $\as X_\alpha \leq n$ uniformly and for each entourage $L\in\mathscr{E}$ there is a subset $Y_L\subseteq X$ with $\as Y_L \leq n$ such that $\{X_\alpha\setminus Y_L\}$ forms an $L$-disjoint collection. Then, $\as X \leq n$.
\end{theorem}

Theorem~\ref{cFADUnion} immediately implies the following version of Lemma~\ref{lem:concat} for asymptotic dimension:

\begin{lemma}\label{FADconcat}
Let $(X,\sE)$ be a coarse space with $\as(X)\le k$ and fix $x_0\in X$. With $X^{(n)}$ as above, $\as(X^{(n)}) \le k$ for all $n\in\NN$. \qed 
\end{lemma}

\begin{lemma}\label{fibering}
If $f:X\to Y$ is a uniformly expansive map of coarse spaces $(X,\sE)$ and $(Y,\sF)$ with $\as Y \le k$ and if coarse fibers of $f$ have asymptotic dimension $n$ uniformly for some $n\in\mathbb{N}$ then $\as X\le (n+1)(k+1)-1$ 
\end{lemma}

\begin{proof}
Let $L\in\sE$ be given. Since $\as Y\le k$, we can cover $Y$ by $k+1$-many $(f\times f)(L)$-disjoint families $\V{0},\ldots,\V{k}$ of uniformly bounded subsets of $Y$. 

Next, for each $V\in \cup_i\V{i}$, since coarse fibers of $f$ have $\as\le n$ uniformly, there is a $K\in\sE$ such that there are uniformly $K$-bounded, $L$-disjoint families $\U{0}^V,\ldots,\U{n}^V$ of subsets of $f^{-1}(V)$, whose union covers $f^{-1}(V)$.

Consider the collection
$\W{i,j}=\{U^V\colon V\in \V{j}, U^V\in \U{i}^V\}.$
We claim that this collection (for $j=0,\ldots,k$ and $i=0,\ldots, n$) is a $K$-uniformly bounded, $L$-disjoint collection of subsets of $X$ that covers $X$.

Since the $\V{j}$ cover $Y$ and the collections $\U{i}^V$ cover $f^{-1}(V)$, it is clear that the collection $\W{i,j}$ covers $X$.

Suppose now that we fix $i_0$ and $j_0$ and consider $\W{i_0,j_0}$. We see that \[\bigcup_{W\in\W{i_0,j_0}}\left(W\times W\right)=\bigcup_{V\in\V{j_0}}\bigcup_{U^V\in\U{i_0}^V}\left(U^V\times U^V\right).\] For each $V$, the union $\bigcup_{U^V\in\U{i_0}^V}\left(U^V\times U^V\right)$ is a subset of $K$. Thus, $$\bigcup_{W\in\W{i_0,j_0}}\left(W\times W\right)$$ is a union of subsets of $K$ and hence a subset of $K$.

Suppose now that $W\neq W'$ in some $\W{i,j}$. We can find $V$ and $V'$ in $\V{j}$ such that $W\in \U{i}^V$ and $W'\in\U{i}^{V'}$. If $V=V'$ then $W\times W'\cap L=\emptyset$ by the assumptions on $\V{j}$. In the case that $V\neq V'$, then $W\times W'\cap L\subset (f\times f)^{-1}(V\times V')\cap (f\times f)^{-1}(f\times f)(L)$. Since $V\times V'\cap (f\times f)(L)=\emptyset$, we see that $(f\times f)^{-1}(V\times V')\cap L$ is also empty. 
\end{proof}

\begin{theorem}
Let $(X,\sE)$ be a coarse space with $\as(X)\le k$ and fix $x_0\in X$. Then $\as(\ast X)\le 2k+1$. 
\end{theorem}

\begin{proof} We take $T$ and $f:\ast X\to T$ as in the proof of Theorem~A.
The map $f$ was shown to be uniformly expansive.
By Lemma \ref{FADconcat}, the coarse fibers of $f$ have asymptotic dimension bounded by $k$ uniformly and so, by Lemma~\ref{fibering} we are done.
\end{proof}

\begin{corollary}\label{lem:bounded-coarse-asdim-1}
Let $(B,\sE)$ be a coarse space with basepoint $x_0$. If $B\times B\in\sE$. Then $\as \ast B\le 1$. \qed 
\end{corollary}

\section{Property C}\label{sec:Property C}

In this section, we adapt the argument that metric free products preserve asymptotic property C~\cite{bell-nagorko2016} to show that coarse property C~\cite{bell-moran-nagorko2016} is preserved by coarse free products. The main obstacle in proving this is that while in the metric case one has a sequence of numbers $R_1,R_2,\ldots$, which have meaning in the metric space $X$ and the metric space $\ast X$, the entourages in the coarse space $X$ do not necessarily have straightforward analogs in the coarse space $\ast X$. 

\begin{definition}\cite{bell-moran-nagorko2016}
A coarse space $X$ has \df{coarse property C} if and only if for any sequence $E_1\subset E_2\subset\cdots$ of entourages there is a finite sequence $\U{1}, \U{2},\ldots,\U{n}$ of subsets of $X$ such that
\begin{enumerate}
\item $\bigcup_{i=1}^n \U{i}$ forms a cover for $X$;
\item each $\U{i}$ is $E_i$-disjoint; and
\item each $\U{i}$ is bounded.
\end{enumerate}
\end{definition}

\begin{theorem}
Let $(X,E)$ be a coarse space. Assume that there is a $k\ge 1$ such that for every infinite sequence $E_1\subset E_2\subset\cdots$ of entourages there is a finite sequence $\U{1},\U{2},\ldots,\U{n}$ of subsets of $X$ such that
\begin{enumerate}
	\item $\bigcup_i\U{i}$ covers $X$;
    \item $\U{i}$ is $E_i$-disjoint; and
    \item $\U{i}$ has asymptotic dimension bounded by $k-1$ uniformly.
\end{enumerate}
Then $X$ has coarse property C.
\end{theorem}

The proof is the same as \cite[Lemma~6.1]{bell-nagorko2016}

\begin{definition}
Let $(X,\sE)$ be a coarse space with basepoint $x_0$.
A subset $A\subset \ast X$ is said to be \df{flat} if there is some $\vec{x}\in\ast X$ such that $A\subset \vec{x}\cdot X$.
Let $E\in\sE$ and $A\subset \ast X$. Define the $E$-\df{cone} $\con_E(A)=A\cdot (\ast B(x_0,E))$, where $B(x_0,E)$ is the symmetric ball $\{x\in X\colon (x_0,x)\in E\cup E^{-1}\}$.
\end{definition}

\begin{lemma} \label{lem:cones-over-flats} Let $\{A_\alpha\}_{\alpha\in J}$ be a collection of uniformly bounded flat subsets of $\ast X$. Then, for each entourage $L\in\sE$, the collection $\{\con_L A_\alpha\}_{\alpha\in J}$ has asymptotic dimension bounded by $1$ uniformly.
\end{lemma}

\begin{proof}
By assumption there is some $K\in\sE$ and an integer $n$ such that $\bigcup_\alpha A_\alpha\times A_\alpha\subset \langle K,n\rangle$. Since each $A_\alpha$ is flat, there is (for each $\alpha$) an $\vec{x}_\alpha\in\ast X$ such that $A_\alpha\subset \vec{x}_\alpha\cdot X$. Therefore, $A_\alpha\subset \vec{x}_\alpha\cdot B(x_0,K^n)$ for each $\alpha$. 

Now, $\con_L A_\alpha\subset \vec{x}_\alpha\cdot B(x_0,K^n)\cdot \ast B(x_0,L)\subset \vec{x}_\alpha\cdot \ast B(x_0,K^n\cup L)$. We apply Corollary~\ref{lem:bounded-coarse-asdim-1} and Lemma~\ref{lem:Guentner} to complete the proof. 

\end{proof}

\begin{definition}
Let $(X,\sE)$ be a coarse space and suppose $E\in\sE$. A set $S\subset X$ is \df{$E$-connected} if for every $x,y\in S$ there is a finite sequence $x = s_0,s_1,\ldots,s_n = y$ of points of $S$ such that $(s_i,s_{i+1})\in E$ for each $i$. An \df{$E$-connected component} of $X$ is a maximal $E$-connected subset of $X$.
\end{definition}

\begin{lemma}
Let $(X,\sE)$ be a coarse space with $x_0\in X$. Suppose that $E\in\ast\sE$. Take some $L\in\sE$ and $n$ such that $E\subset\langle L,n\rangle$ and suppose that $A\subset \ast X\cdot\left(X\setminus B(x_0,L^n)\right)$ has the property that $E$-connected components of $A$ are uniformly bounded. Then, for each $M\in\sE$, the $E$-connected components of $\con_MA$ have asymptotic dimension at most $1$ uniformly.
\end{lemma}

\begin{proof} We prove this first under the assumption that $E=\langle L,n\rangle$. The general case follows from the fact that $E$-connected components are contained in some $\langle L,n\rangle$-connected component and the fact that asymptotic dimension is monotonic on subsets.

Following the method of \cite[Lemma 6.11]{bell-nagorko2016}, we can characterize the $E$-connected components of $\con_M(A)$ as follows:
$C$ is an $E$-connected component of $\con_MA$ if and only if $X^{(\le n)}\cap C$ is an $E$-connected component of $X^{(\le n)}\cap \con_MA$. 

Let $C$ be an $E$-connected component of $\con_MA$. Put $C_k=C\cap X^{(\le k)}$. Let $k_0$ be the smallest integer for which $C_{k_0}\neq \emptyset$. We claim that $C_{k_0}$ is flat and uniformly bounded.

Take two words $\vec{x}$ and $\vec{y}$ in $C_{k_0}$. Then, $\vec{x}$ and $\vec{y}$ are in $\con_MA$ and there is an $E$-chain $\left(\vec{t}_i\right)$ of elements of $C$ connecting them. By our observation above, we may take the $\vec{t}_i$ in $C_{k_0}$.

Write $\vec{t}_1=\vec{a}_t\cdot b_t z_l\cdots z_{k_0}$ with $\vec{a}_{t}\cdot b_t\in A$, and $b_t\in X\setminus B(x_0,L^n)$. Similarly, write $\vec{x}=\vec{a}_x\cdot b_x x_{l}\cdots x_{k_0}$. Since $(\vec{x},\vec{t}_1)\in E$, we have $D^\ast_L(\vec{x},\vec{t}_1)\le n$. Since $\|b_x\|>n$ and $\|b_t\|>n$, we see that $\vec{a}_x=\vec{a}_t$, and in particular, 
\begin{equation}\label{eqn:1} n\ge D^\ast_L(\vec{x},\vec{t}_1)=D_L(b_x,b_t)+\sum_{j=l}^{k_0}\left(\|x_j\|_L+\|z_j\|^L\right).\end{equation} 

Next, we suppose that the words $z_l\cdots z_{k_0}$ and $x_{l}\cdots x_{k_0}$ are non-empty. Then, the minimality of $k_0$ means that $(\vec{a}b_xx_l\cdots x_{k_0-1}, \vec{x})\notin E$. Thus, $\|x_{k_0}\|_L>n$, contradicting Equation~\eqref{eqn:1}. We conclude that $C_{k_0}\subset A$ and is therefore uniformly bounded. We conclude also that $\vec{a}_x=\vec{a}_y$ and that $C_{k_0}$ is therefore flat.

We show by induction on $k$ that \[C_k\subset \con_{M\cup L^n\cup D} C_{k_0} \]

Let $\vec{x}\in C_{k+1}\setminus C_k$, with $k\ge k_0$. Then, either $\vec{x}\in \con_M C_k $ or by the argument above, $\vec{x}$ lies in some $E$-connected component of $A$ that is also $E$-close to $\con_M C_k$. In the first case, we see that $\vec{x}\in \con_{M\cup L^n\cup D}C_k$. In the second case, if $D\in\sE$ is a bound on the diameter of $A$, $\vec{x}\in\con_{M\cup L^n\cup D}C_k$. Since $\con_{M\cup L^n\cup D}\con_{M\cup L^n\cup D}C_k=\con_{M\cup L^n\cup D}C_k$, we have proved our claim. By Lemma~\ref{lem:cones-over-flats}, the $E$-connected components of $\con_M A$ have asymptotic dimension at most $1$ uniformly. 
\end{proof}

\begin{theorem} Let $X$ be a coarse space with fixed basepoint $x_0$. If $X$ has coarse property $C$, then $\ast X$ has coarse property C.
\end{theorem}

\begin{proof} Suppose $E_1\subset E_2\subset \cdots$ is a given sequence of entourages in $\ast \sE$. For each $i$ find $L_i\in \sE$ and an integer $n_i$ such that $E_i\subset\langle L_i,n_i\rangle$.  Find a sequence $\U{1},\U{2},\ldots, \U{p}$ of $L_i^{n_i}$-disjoint, uniformly bounded subsets of $X$ whose union covers $X$. Put $n=\max\{n_i\}$.

Put $\V{i}(\vec{x})=\{\vec{x}\cdot (U\setminus B(x_0,L^n_{p+1})\colon U\in\U{i}, \vec{x}\in\ast X\}$ for each $i\in\{1,2,\ldots,p\}$. Put $\V{p+1}=\{\{x_0\}\}$. We claim that 
\begin{enumerate}
\item $\bigcup_{i=1}^{p+1} \con_{L^n_{p+1}}\cup \V{i}=\ast X$ and
\item $\{\V{i}(\vec{x})\}_{i,\vec{x}}$ is $E_i$-disjoint, uniformly bounded, and its elements are flat.
\end{enumerate}

For (1), we consider an element $\vec{x}\in \ast X$. Write $\vec{x}=x_1x_2\cdots x_k$, where each $x_i\in X$. If $(x_0,x_i)\in L_{p+1}^n$ for each $i$, then $\vec{x}\in \con_{L^n_{p+1}}\{x_0\}$. Otherwise, take $m$ to be the largest integer for which $(x_0,x_m)\notin L_{p+1}^n$. Find some $U$ in some $\U{l}$ such that $x_m\in U$. Then, $x\in x_1x_2\cdots x_{m-1}\cdot (U\setminus B(x_0,L_{p+1}^n))\cdot \ast B(x_0,L_{p+1}^n)$. Thus, $\vec{x}\in \con_{L^n_{p+1}}\cup \V{l}$.

For (2), we observe that each $\V{i}$ is uniformly bounded and flat, so it remains only to show that these families are $E_i$-disjoint. Suppose that $V_1$ and $V_2$ are distinct elements of some $\V{i}$. We can find $\vec{x}_1,\vec{x}_2\in \ast X$ and subsets $U_1$ and $U_2$ in $\U{i}$ for which $V_1=\vec{x}_1\cdot \tilde{U_1}$ and $V_2=\vec{x}_2\cdot \tilde{U_2}$, where $\tilde{U_i}$ denotes the set $U_i$ with the ball $B(x_0,L_{p+1}^n)$ removed. If $\vec{x}_1=\vec{x}_2$, then we're done since $\U{i}$ is $L_{i}^{n_i}$-disjoint.

Otherwise, take $\vec{v}_1=\vec{x}_1u_1$ and $\vec{v}_2=\vec{x}_2u_2$. We rewrite $\vec{v}_1 = (\vec{v}_1\wedge \vec{v}_2)b\vec{c}u_1$ and $\vec{v}_2 = (\vec{v}_1\wedge \vec{v}_2)b'\vec{c}'u_2$. We compute \begin{align*}D_{L_{i}}^*(\vec{v}_1,\vec{v}_2) &= D_{L_{i}^{n_i}}(b,b') + \|\vec{c}\|_{L_i} + \|\vec{c}'\|^{L_i} + D_{L_i}(x_0,u_1) + D_{L_i}(u_2,x_0)\\
&\ge D_{L_{i}^{n_i}}(b,b') + \|\vec{c}\|_{L_i} + \|\vec{c}'\|^{L_i} + D_{L_i}(u_1,u_2) \\ & \ge n_i.\end{align*}
Thus we see $\V{i}$ is $\langle L_i,n_i\rangle$-disjoint hence $E_i$-disjoint.

\end{proof}

\bibliographystyle{abbrv}

\end{document}